\documentclass[10pt]{amsart}
\usepackage{verbatim,hyperref}
\usepackage{rotating} 
\usepackage{amssymb}
\usepackage{url}
\usepackage{xcolor,xypic}

\newtheorem{theorem}{Theorem}[section]

\newtheorem{prop}[theorem]{Proposition}

\newtheorem{lemma}[theorem]{Lemma}
\newtheorem{fact}[theorem]{Fact}

\newtheorem{definition}[theorem]{Definition}

\newtheorem{conjecture}[theorem]{Conjecture}

\theoremstyle{plain}
\numberwithin{equation}{theorem}

\theoremstyle{remark}
\newtheorem{claim}[theorem]{Claim}

\newtheorem{remark}[theorem]{Remark}

\newcommand{\Fp}{{\mathbb F}_p}
\newcommand{\Q}{{\mathbb Q}}

\newcommand{\Z}{{\mathbb Z}}

\newcommand{\OO}{{\mathcal O}}

\renewcommand{\L}{{\mathcal L}}

\newcommand{\id}{\operatorname{id}}

\newcommand{\Qbar}{\overline{\Q}}

\DeclareMathOperator{\im}{Im}

\DeclareMathOperator{\Gal}{Gal}
\DeclareMathOperator{\GL}{GL}

\DeclareMathOperator{\N}{\mathbb{N}}

\DeclareMathOperator{\tor}{tor}
\DeclareMathOperator{\tors}{tors}

\DeclareMathOperator{\End}{End}

\DeclareMathOperator{\Pic}{Pic}

\newcommand{\isomto}{\overset{\sim}{\rightarrow}}

\newcommand{\bG}{{\mathbb G}}

\newcommand{\bA}{{\mathbb A}}

\newcommand{\lra}{\longrightarrow}

\newif\ifhascomments \hascommentstrue
\ifhascomments
  \newcommand{\dragos}[1]{{\color{red}[[\ensuremath{\bigstar\bigstar\bigstar} #1]]}}
  \newcommand{\matt}[1]{{\color{red}[[\ensuremath{\spadesuit\spadesuit\spadesuit} #1]]}}
\else
  \newcommand{\dragos}[1]{}
  \newcommand{\matt}[1]{}
\fi

\begin{document}

\title[Density of orbits inside semiabelian varieties]{Density of orbits of dominant regular self-maps of semiabelian varieties}

\author{Dragos Ghioca}
\author{Matthew Satriano} 

\thanks{The authors were partially supported by Discovery grants from the National Sciences and Engineering Research Council of Canada.}

\address{Department of Mathematics \\
University of British Columbia\\
Vancouver, BC V6T 1Z2\\
Canada}
\email{dghioca@math.ubc.ca}

\address{Pure Mathematics \\ 
University of Waterloo \\
200 University Avenue West \\
Waterloo, Ontario, Canada N2L 3G1   
}
\email{msatriano@uwaterloo.ca}

\begin{abstract} We prove a conjecture of Medvedev and Scanlon \cite{medvedev-scanlon} in the case of regular morphisms of semiabelian varieties. That is, if $G$ is a semiabelian variety defined over an algebraically closed field $K$ of characteristic $0$, and $\varphi\colon G\to G$ is a dominant regular self-map of $G$ which is not necessarily a group homomorphism, we prove that one of the following holds: either there exists a non-constant rational fibration preserved by $\varphi$, or there exists a point $x\in G(K)$ whose $\varphi$-orbit is Zariski dense in $G$.
\end{abstract}

\maketitle


\section{Introduction}
\label{intro section}

For any self-map $\Phi$ on a set $X$, and any non-negative integer $n$, we denote by $\Phi^n$ the $n$-th compositional power, where $\Phi^0$ is the identity map. For any $x\in X$, we denote by $\OO_\Phi(x)$ its orbit under the action of $\Phi$, i.e., the set of all iterates $\Phi^n(x)$ for $n\ge 0$. 

Our main result is the following.
\begin{theorem}
\label{main result}
Let $G$ be a semiabelian variety defined over an algebraically closed field $K$ of characteristic $0$ and let $\varphi:G\lra G$ be a dominant regular self-map which is not necessarily a group homomorphism. Then either there exists $x\in G(K)$ such that $\OO_\varphi(x)$ is Zariski dense in $G$, or there exists a nonconstant rational function $f\in K(G)$ such that $f\circ \varphi = f$. 
\end{theorem}

Theorem~\ref{main result} answers affirmatively the following conjecture raised by Medvedev and Scanlon in \cite{medvedev-scanlon} for the case of regular morphisms of semiabelian varieties.

\begin{conjecture}[{\cite[Conjecture~7.14]{medvedev-scanlon}}]
\label{M-S conjecture}
Let $X$ be a quasiprojective variety defined over an algebraically closed field $K$ of characteristic $0$ and let $\varphi:X\dashrightarrow X$ be a rational self-map. Then either there exists $x\in X(K)$ whose orbit under $\varphi$ is Zariski dense in $X$, or $\varphi$ preserves a nonconstant fibration, i.e., there exists a nonconstant rational function $f\in K(X)$ such that $f\circ \varphi = f$.
\end{conjecture}

The origin of \cite[Conjecture~7.14]{medvedev-scanlon} lies in a much older conjecture formulated by Zhang in the early 1990's (and published in \cite[Conjecture~4.1.6]{Zhang-lec}). Zhang asked that for each polarizable endomorphism $\varphi$ of a projective variety $X$ defined over $\Qbar$ there must exist a $\Qbar$-point with Zariski dense orbit under $\varphi$. Medvedev and Scanlon \cite{medvedev-scanlon} conjectured that as long as $\varphi$ does not preserve a nonconstant fibration, then a Zariski dense orbit must exist; the hypothesis concerning polarizability of $\varphi$ already implies that no nonconstant fibration is preserved by $\varphi$. In \cite{medvedev-scanlon}, they also prove their conjecture in the special case $X=\bA^n$ and $\varphi$ is given by the coordinatewise action of $n$ one-variable polynomials $(x_1,\dots, x_n)\mapsto (f_1(x_1),\dots, f_n(x_n))$; their result was established over an arbitrary field $K$ of characteristic $0$ which is not necessarily algebraically closed. 

In \cite{Amerik-Campana}, Amerik and Campana proved Conjecture~\ref{M-S conjecture} for all uncountable algebraically closed fields $K$ (see also \cite{BRS} for a proof of the special case of this result when $\varphi$ is an automorphism). In fact, Conjecture~\ref{M-S conjecture} is true even in positive characteristic, as long as the field $K$ is uncountable (see \cite[Corollary~6.1]{preprint}); on the other hand, when the transcendence degree of $K$ over $\Fp$ is smaller than the dimension of $X$, there are counterexamples to the corresponding variant of Conjecture~\ref{M-S conjecture} in characteristic $p$ (as shown in \cite[Example~6.2]{preprint}). 

With the notation as in Conjecture~\ref{M-S conjecture}, it is immediate to see that if $\varphi$ preserves a nonconstant fibration, then there is no Zariski dense orbit. So, the real difficulty in Conjecture~\ref{M-S conjecture} lies in finding a Zariski dense orbit for a self-map $\varphi$ of $X$, when the algebraically closed field $K$ is \emph{countable}; in this case, there are only a handful of results known, as we will briefly describe below.
\begin{itemize}
\item In \cite{Amerik-Bogomolov-Rovinsky}, Conjecture~\ref{M-S conjecture} was proven assuming there is a point $x\in X(K)$ which is fixed by $\varphi$ and moreover, the induced action of $\varphi$ on the tangent space of $X$ at $x$ has multiplicatively independent eigenvalues.
\item Conjecture~\ref{M-S conjecture} is known for varieties $X$ of positive Kodaira dimension, see for example \cite[Proposition~2.3]{preprint 2}.
\item In \cite{Xie-1}, Conjecture~\ref{M-S conjecture} was proven for all birational automorphisms of surfaces (see also \cite{BGT} for an independent proof of the case of automorphisms). Later, Xie \cite{Xie-2} established the validity of Conjecture~\ref{M-S conjecture} for all polynomial endomorphisms of $\bA^2$.
\item In \cite{preprint 2}, the conjecture was proven all smooth minimal $3$-folds of Kodaira dimension $0$ with sufficiently large Picard number, contingent on certain conjectures in the Minimal Model Program.
\item In \cite{GS-abelian}, Conjecture~\ref{M-S conjecture} was proven for all abelian varieties.
\item In \cite{G-X}, it was proven that if Conjecture~\ref{M-S conjecture} holds for the dynamical system $(X,\varphi)$, then it also holds for the dynamical system $(X\times \bA^n, \psi)$, where $\psi:X\times \bA^n\dashrightarrow X\times \bA^n$ is given by $(x,y)\mapsto (\varphi(x), A(x)y)$ and $A\in {\rm GL}_n(K(X))$.
\end{itemize}

Our Theorem~\ref{main result} extends the main result of \cite{GS-abelian} where Conjecture~\ref{M-S conjecture} was shown for abelian varieties. There are numerous examples in arithmetic geometry when one needed to overcome significant difficulties to extend a known result for abelian varieties to the case of semiabelian varieties: the case of non-split semiabelian varieties presented intrinsic complications in each of the classical conjectures of Mordell-Lang, Bogomolov, or Pink-Zilber. In the case of the Medvedev-Scanlon conjecture, the major technical obstacle we face is the absence of Poincar\'e's Reducibility Theorem: if $A$ is an abelian variety and $B\subset A$ is an abelian subvariety, then there exists an abelian subvariety $C\subset A$ such that $A=B+C$ and $B\cap C$ is finite, i.e.~$A/B$ is isogenous to an abelian subvariety of $A$. The corresponding version of this result is false for semiabelian varieties. Since Poincar\'e's Reducibility Theorem is used throughout \cite{GS-abelian}, our proof of Theorem \ref{main result} requires significant conceptual changes, specifically in the proofs of the main results of Subsections \ref{Minimal dominating semiabelian subvarieties}, \ref{Constructing topological generators}, and \ref{section proof}. Also, the absence of Poincar\'e's Reducibility Theorem in the case of non-split semiabelian varieties $G$ makes it impossible for one to use a similar strategy as in \cite{GS-abelian} in order to prove a generalization of Theorem~\ref{main result} when the action of $\varphi$ is replaced by the action of a finitely generated commutative monoid of regular self-maps on $G$; for more details, see Remark~\ref{commutative monoid}.

The plan of our paper is as follows. In Section~\ref{section semiabelian} we introduce our notation and state the various useful facts about semiabelian varieties which we will employ in our proof. We continue in Section~\ref{section results} by proving several reductions and auxilliary statements to be used in the proof of our main result. Finally, we conclude by proving Theorem~\ref{main result} in Section~\ref{section proof}.


\section{Properties of semiabelian varieties}
\label{section semiabelian}

\subsection{Notation} 
We start by introducing the necessary notation for our paper. 

Let $G_1$ and $G_2$ be abelian groups and let $G=G_1\times G_2$. By an abuse of notation, we identify $G_1$ as a subgroup of $G$ through the inclusion map $x\mapsto (x,0)$; similarly, we identiy $G_2$ with a subgroup of $G$ through the inclusion map $x\mapsto (0,x)$. Also, viewing $G$ as $G_1\oplus G_2$, then for any $x_1\in G_1$ and $x_2\in G_2$, we use either the notation $(x_1,x_2)$ or $x_1\oplus x_2$ for the element $(x_1,x_2)\in G$. 
For any group $G$, we denote by $G_{\tor}$ its torsion subgroup; also, if $G$ is abelian, then (unless otherwise noted) we denote its group operation by ``$+$''.

\subsection{Semiabelian varieties} 
We continue by stating some useful facts regarding semiabelian varieties. Unless, otherwise noted, $G$ denotes a semiabelian variety defined over an algebraically closed field $K$ of characteristic $0$.

The following structure result for regular self-maps on semiabelian varieties is proven in \cite[Theorem~5.1.37]{Noguchi-Winkelmann}.  
\begin{fact}
\label{rigidity}
Let $G_1$ and $G_2$ be  semiabelian varieties and let 
$\varphi:G_1\lra G_2$. Then there exists a group homomorphism $\tau:G_1\lra G_2$ and there exists $y\in G_2$ such that $\varphi(x)=\tau(x)+y$ for each $x\in G_1$.
\end{fact}

By definition (see \cite[Definition~5.1.20]{Noguchi-Winkelmann} and \cite[Fact~2.4]{Benoist}), a semiabelian variety over $K$ is a commutative algebraic group $G$ over $K$ for which there is 
an algebraic torus $T$, an abelian variety $A$, and a short exact sequence of algebraic groups over $K$:
\begin{equation}
\label{presentation semiabelian}
0\lra T\lra G\lra A\lra 0
\end{equation}
We often say that $T$ is the \emph{toric part} of $G$, while $A$ is the \emph{associated abelian variety} of $G$. When the short exact sequence \eqref{presentation semiabelian} splits, we say that $G$ is a \emph{split} semiabelian variety.

The next fact will be used several times in our proof.

\begin{fact}
\label{torus not abelian}
There is no nontrivial group homomorphism between an algebraic torus and an abelian variety. 
\end{fact}

As a consequence, we have the following: suppose $\sigma\colon G_1\lra G_2$ is a group homomorphism of semiabelian varieties and
\[
0\lra T_i\lra G_i\stackrel{p_i}{\lra} A_i\lra 0
\]
is a short exact sequence with $T_i$ the toric part of $G_i$ and $A_i$ the associated abelian variety of $G_i$. Then $p_2(\sigma(T_1))=0$, so we have:

\begin{fact}
\label{morphism fact}
Let $G_1$ and $G_2$ be semiabelian varieties with toric parts $T_1$ and $T_2$, respectively associated abelian varieties $A_1$ and $A_2$. Then for any group homomorphism $\sigma:G_1\lra G_2$, the restriction $\sigma|_{T_1}$ induces a group homomorphism between $T_1$ and $T_2$; furthermore there is an induced group homomorphism $\overline{\sigma}:A_1\lra A_2$. 
\end{fact}

Thus, we see that morphisms of semiabelian varieties induce morphisms of their corresponding tori and associated abelian varieties. There is a converse to this statement as well. If $G$ is a semiabelian variety and $p\colon G\to A$ is the quotient map to its associated abelian variety, then $p$ is a $T$-torsor, and hence $G$ is the relative spectrum of $\bigoplus_{m\in M}\L_m$ where $\L_m$ are line bundles and $M$ is the character lattice of $T$. One shows, see e.g.~\cite[Corollary 3.1.4.4]{kai-wen-thesis}, that for all $m,m'\in M$ we have $\L_m\otimes\L_{m'}\simeq\L_{m+m'}$ and each $\L_m\in\Pic^0(A)=A^\vee$. In other words, we have a group homomorphism $c\colon M\to A^\vee$. If $\sigma\colon G'\to G$ is a group homomorphism, then from Fact~\ref{morphism fact} we have homomorphisms $\sigma|_{T'}\colon T'\to T$ and $\psi=\overline\sigma\colon A'\to A$ between the toric parts and associated abelian varieties. This in turn, induces a homomorphism $\phi\colon M\to M'$ between the character lattices of $T$ and $T'$, and a homomorphism $\psi^\vee\colon A^\vee\to (A')^\vee$ between dual abelian varieties. Via these constructions, we obtain an equivalence of categories:

\begin{fact}[{\cite[Proposition 3.1.5.1]{kai-wen-thesis}}]
\label{antiequiv of cats}
The category of semiabelian varieties is anti-equivalent to the following category: objects are group homomorphisms $c\colon M\to A^\vee$ where $M$ is a finitely generated free abelian group and $A$ is an abelian variety;
morphisms of objects $(c\colon M\to A^\vee)\to(c'\colon M'\to (A')^\vee)$ consist of commutative diagrams
\[
\xymatrix{
M\ar[r]^-{c}\ar[d]^-{\phi} & A^\vee\ar[d]^-{\psi^\vee}\\
M'\ar[r]^-{c'} & (A')^\vee
}
\]
where $\phi$ is a group homomorphism and $\psi\colon A'\to A$ is a homomorphism of abelian varieties.
\end{fact}

From Fact~\ref{antiequiv of cats}, we see that if $G$ is a semiabelian variety corresponding to the homomorphism $c\colon M\to A^\vee$, then $\End(G)$ is the subring of $\End(T)\times\End(A)$ consisting of pairs $(\alpha,\psi)$ such that $c\circ\alpha^\vee=\psi^\vee\circ c$, where $\alpha^\vee\in\End(M)$ is the endomorphism of the character lattice induced by $\alpha$. So we have the following.

\begin{fact}
\label{endomorphism f.g.}
With the notation as in \eqref{presentation semiabelian}, we let $\End(T)$, $\End(G)$, and $\End(A)$ be the endomorphism rings of the corresponding algebraic groups. Then the endomorphism ring $\End(G)$ embeds into $\End(T)\times \End(A)$. In particular, $\End(G)$ is a finitely generated $\Z$-module. 
\end{fact}

\begin{fact}
\label{fact minimal poly}
Let $G$ be a semiabelian variety and $\varphi:G\lra G$ be a group homomorphism. Then there exists a monic polynomial $f\in\Z[z]$ of degree at most equal to $2\dim(G)$ such that $f(\varphi(x))=0$ for all $x\in G$. 

Moreover, for any $x\in G(K)$ and any regular self-map $\varphi:G\lra G$, the orbit  $\OO_{\varphi}(x)$ is contained in a finitely generated subgroup of $G$.
\end{fact}

\begin{proof}
For the first part, by Fact~\ref{endomorphism f.g.} it is enough to show that each $(\phi,\psi)\in\End(T)\times\End(A)$ satisfies a monic polynomial of degree at most $2\dim(G)$. Letting $d=\dim(T)$, we have $\End(T)\isomto M_d(\Z)$ the ring of $d$-by-$d$ matrices with integer entries. Then the matrix corresponding to $\phi$ satisfies its characteristic polynomial $g(z)$ which has degree $d$. By \cite[Fact~3.3]{GS-abelian}, we know that $\psi$ satisfies a monic polynomial $h(z)$ of degree at most $2\dim(A)$, so we can take $f=gh$.

We now prove the ``moreover'' statement. By Fact~\ref{morphism fact}, there exists $y\in G(K)$ and $\tau\in\End(G)$ such that $\varphi(x)=\tau(x)+y$ for any $x\in G$. Then for all $n\in\N$, we have
$$\varphi^n(x)=\tau^n(x)+y+\tau(y)+\cdots + \tau^{n-1}(y).$$
Since there exists a monic polynomial $f\in\Z[z]$ of degree at most $2\dim(G)$ such that $f(\tau)=0$, we conclude that $\OO_\varphi(x)$ is contained in the finitely generated subgroup of $G(K)$ spanned by $\tau^i(x)$ and $\tau^i(y)$ for $0\le i\le 2\dim(G)-1$. 
\end{proof}

For each positive integer $n$, we let $G[n]$ be the group of torsion points of $G$ killed by the multiplication-by-$n$ map on $G$. Then, as shown in \cite[Fact~2.9]{Benoist}, $G[n]\isomto (\Z/n\Z)^{\dim(T)+2\dim(A)}$ where $T$ and $A$ are the toric part, respectively the associated abelian variety of $G$; see \eqref{presentation semiabelian}. Therefore, similar to the case of abelian varieties (see \cite[Fact~3.10]{GS-abelian}), we obtain the following result.

\begin{fact}
\label{torsion extensions}
Let $G$ be a semiabelian variety defined over a field $K_0$ of characteristic $0$. 
Then the group $\Gal(K_0(G_{\tor})/K_0)$ embeds as a closed subgroup of $\GL_{\dim(T)+2\dim(A)}(\widehat{\Z})$, where $T$ and $A$ are the toric part, respectively the associated abelian variety of $G$, and $\widehat{\Z}$ is 
the ring of finite ad\'eles.
\end{fact}

The following result, proven by Faltings \cite{Faltings} for abelian varieties and by Vojta \cite{Vojta} for semiabelian varieties, was known as the Mordell-Lang conjecture.
\begin{fact}[Vojta \cite{Vojta}]
\label{Faltings theorem}
Let $V\subset G$ be an irreducible subvariety of the semiabelian variety $G$ defined over an algebraically closed field $K$ of characteristic $0$. Assume there exists a finitely 
generated subgroup $\Gamma\subset G(K)$ such that $V(K)\cap\Gamma$ is Zariski dense in 
$V$. Then $V$ is a coset of a semiabelian subvariety of $G$.
\end{fact}

Combining Fact~\ref{fact minimal poly} with Fact~\ref{Faltings theorem}, we obtain:

\begin{fact}
\label{fact closure orbit}
Let $\varphi:G\lra G$ be a self-map and let $x\in G(K)$. The Zariski closure of $\OO_{\varphi}(x)$ is a finite union of cosets of semiabelian subvarieties of $G$. 
\end{fact}

\begin{proof}
Using the ``moreover'' part provided by  Fact~\ref{fact minimal poly}, we see $\OO_{\varphi}(x)$ is contained in a finitely generated subgroup $\Gamma$ of $G(K)$. Letting $V$ be the closure of $\OO_{\varphi}(x)$ we see $V(K)\cap \Gamma$ is Zariski dense in $V$. Fact~\ref{Faltings theorem} then tells us that each irreducible component of $V$ is a coset of a semiabelian subvariety of $G$, which finishes the proof.
\end{proof}

Finally, we end with the following easy observation which will be used in Section~\ref{section results}.
\begin{fact}
\label{torus not abelian 2}
Let
\[
0\lra T\lra G\stackrel{p}{\lra} A\lra 0
\]
be a short exact sequence of algebraic groups with $T$ being a torus and $A$ an abelian variety. 
If $H\subset G$ is an algebraic subgroup such that $A=p(H)$, then $G/H$ is an algebraic torus. 
\end{fact}
\begin{proof}
We obtain the following diagram where the rows are short exact sequences and the vertical arrows are inclusions:
\[
\xymatrix{
1\ar[r] & H\cap T\ar[r]\ar@{^{(}->}[d] & H\ar[r]\ar@{^{(}->}[d] & H/(H\cap T)\ar[r]\ar@{^{(}->}[d] & 1\\
1\ar[r] & T\ar[r] & G\ar[r]^-{p} & A\ar[r] & 1
}
\]
Since $p(H)=A$, we see $H/(H\cap T)=A$. So, we have an isomorphism $T/(H\cap T)\simeq G/H$ which finishes the proof since quotients of tori are tori.
\end{proof}


\section{Useful results}
\label{section results}

In following subsections, we prove several propositions which will then be used in order to derive Theorem~\ref{main result}.

\subsection{Minimal dominating semiabelian subvarieties}
\label{Minimal dominating semiabelian subvarieties}

\begin{lemma}
\label{replacing by a conjugate}
It suffices to prove Theorem~\ref{main result} for a conjugate $\sigma^{-1}\circ \varphi\circ \sigma$ of the self-map $\varphi:G\lra G$ under some automorphism $\sigma\colon G\lra G$.
\end{lemma}

\begin{proof}
This is \cite[Lemma~5.4]{GS-abelian}; the proof goes verbatim not only when $G$ is a semiabelian variety, but also for any quasiprojective variety.
\end{proof}

\begin{definition}
\label{def:min dominating subvariety}
Let $G$ be a semiabelian variety and
\begin{equation}
\label{presentation semiabelian 2}
0\lra T\lra G\stackrel{p}{\lra} A\lra 0
\end{equation}
the corresponding short exact sequence. We say $H\subset G$ is a \emph{minimal dominating semiabelian subvariety of} $G$ if: (i) $H$ is a semiabelian subvariety with $p(H)=A$ and (ii) for any semiabelian subvariety $H'\subset G$ with $p(H')=A$, we have $H\subset H'$.
\end{definition}

We show the existence of minimal dominating semiabelian subvarieties, after allowing for an isogeny.

\begin{lemma}
\label{min dominant subgp exists after finite cover}
For every semiabelian variety $G$, there exists an isogeny $f\colon G'\to G$ such that $G'$ has a minimal dominating semiabelian subvariety. 

Moreover, if $G=G_1\times G_2$ with the $G_i$ semiabelian varieties, then there exist isogenies $f_i\colon G'_i\to G_i$ such that $G'_1\times G'_2$ has a minimal dominating semiabelian subvariety.
\end{lemma}
\begin{proof}
By Fact \ref{antiequiv of cats}, the semiabelian variety $G$ corresponds to a morphism $c\colon M\to A^\vee$ where $M$ is the character lattice of $T$. To begin, notice that a semiabelian subvariety $H_0\subset G$ has $p(H_0)=A$ (see \eqref{presentation semiabelian 2}) if and only if it induces a diagram
\[
\xymatrix{
1\ar[r]& T_0\ar[r]\ar@{^{(}->}[d]& H_0\ar[r]\ar@{^{(}->}[d]& A\ar[r]\ar@{=}[d]& 1\\
1\ar[r] & T\ar[r] & G\ar[r] & A\ar[r] & 1
}
\]
where the rows are short exact. By Fact \ref{antiequiv of cats}, this is equivalent to factoring $c$ as $M\to M_0\to A^\vee$ with $M\to M_0$ a surjection of free abelian groups. Therefore, a minimal dominating semiabelian subvariety exists if and only if $c$ factors as $M\stackrel{q}{\lra}\overline{M}\stackrel{\overline{c}}{\lra} A^\vee$ such that (i) $q$ is a surjection of free abelian group, and (ii) for all factorizations $M\stackrel{q_0}{\lra}M_0\stackrel{c_0}{\lra} A^\vee$ of $c$, there exists a surjection $q_1\colon M_0\to\overline{M}$ such that $q=q_1\circ q_0$ and $c_0=\overline{c}\circ q_1$. In particular, if the image $\im(c)$ is torsion-free, then a minimal dominating semiabelian subvariety exists.

Since $\im(c)\subset A^\vee$ is a subgroup, we see that the torsion part $\im(c)_{\tors}$ is a finite subgroup of $A^\vee$. Let $\Gamma$ be any finite subgroup of $A^\vee$ that contains $\im(c)_{\tors}$. Then $\im(c)_{\tors}=\im(c)\cap\Gamma$. Since $\Gamma$ is a finite subgroup, $\pi\colon A^\vee\to A^\vee/\Gamma$ is an isogeny of abelian varieties, and by construction the image of the map $\pi\circ c\colon M\to A^\vee/\Gamma$ is equal to $\im(c)/\im(c)_{\tors}$ which is torsion-free. Letting $A'=(A^\vee/\Gamma)^\vee$ and $\psi=\pi^\vee$, we have $\pi=\psi^\vee$ and $\psi\colon A'\to A$ is an isogeny, see e.g.~\cite[Theorem 9.1]{Milne}. By Fact \ref{antiequiv of cats}, we have a morphism of short exact sequences
\[
\xymatrix{
1\ar[r]& T\ar[r]\ar@{=}[d]& G'\ar[r]\ar[d]^-{f}& A'\ar[r]\ar[d]^-{\psi}& 1\\
1\ar[r] & T\ar[r] & G\ar[r] & A\ar[r] & 1
}
\]
where $G'$ is defined by $\pi\circ c\colon M\to A^\vee/\Gamma=(A')^\vee$. We see then that $f$ is an isogeny. Since the image of $\pi\circ c$ is torsion-free, $G'$ has a minimal dominating semiabelian subvariety.

Finally, it remains to handle the case when $G=G_1\times G_2$. Here, $G_i$ is defined by a map $c_i\colon M_i\to A_i^\vee$ with $M_i$ finitely generated free abelian groups and $A_i$ abelian varieties. Then $G$ is defined by the map $c=(c_1,c_2)\colon M_1\oplus M_2\to A_1^\vee\times A_2^\vee=(A_1\times A_2)^\vee$. Then $\im(c)=\im(c_1)\oplus\im(c_2)$ so $\im(c)_{\tors}=\im(c_1)_{\tors}\oplus\im(c_2)_{\tors}$. We can then choose $\Gamma=\Gamma_1\times\Gamma_2\subset A_1^\vee\times A_2^\vee$ where $\Gamma_i\subset A_i^\vee$ is a finite subgroup containing $\im(c_i)$. The resulting isogeny $\psi\colon A'\to A_1\times A_2$ defined by $\Gamma$ in the previous paragraph is then of the form $\psi=\psi_1\times\psi_2$ where $\psi_i\colon A'_i\to A_i$ is the isogeny defined by $\Gamma_i$.
\end{proof}

\begin{lemma}
\label{conditions for containing G1}
For $i=1,2$, let $G_i$ be a semiabelian variety fitting into a short exact sequence
\[
0\lra T_i\lra G_i\stackrel{p_i}{\lra} A_i\lra 0
\]
with $T_i$ a torus and $A_i$ an abelian variety. Let $G=G_1\times G_2$ and $p=(p_1,p_2)\colon G\lra A_1\times A_2$. If $H\subset G$ is an algebraic subgroup with $T_1\subset H$ and $p(H)=A_1\times A_2$, then $G_1\subset H$.
\end{lemma}
\begin{proof}
To prove the lemma, it suffices to replace $H$ by the connected component of the identity of $H$, and so we can assume $H$ is a semiabelian subvariety of $G$. By Fact \ref{antiequiv of cats}, we know that $G_i$ corresponds to a group homomorphism $c_i\colon M_i\lra A_i^\vee$ where $M_i$ is the character lattice of $T_i$. Then $G$ corresponds to the homomorphism $c=(c_1,c_2)\colon M_1\oplus M_2\lra A_1^\vee\times A_2^\vee$. Since $H$ is a semiabelian subvariety of $G$ and $p(H)=A_1\times A_2$, then as in the proof of Lemma \ref{min dominant subgp exists after finite cover}, we know that $H$ corresponds to a factorization $c'$ of $c$ through a quotient of $M_1\oplus M_2$. Moreover since $T_1\subset H$, the quotient is of the following form: there is a surjection $\pi\colon M_2\lra M'$ and $H$ corresponds to a group homomorphism $c'\colon M_1\oplus M'\lra A_1^\vee\times A_2^\vee$ such that $c=c'\circ(\id,\pi)$ where $\id$ is the identity map on $M_1$. Consider the following diagram
\[
\xymatrix{
M_1 \oplus M_2 \ar[r]^-{c} \ar[d]^-{(\id,\pi)} & A_1^\vee \times A_2^\vee\ar@{=}[d] \\
M_1 \oplus M' \ar[r]^-{c'} \ar[d]^-{\pi'} & A_1^\vee \times A_2^\vee\ar[d]^-{q} \\
M_1  \ar[r]^-{c_1} & A_1^\vee
}
\]
where $\pi'$ and $q$ are the natural projections. Since $(\id,\pi)$ is surjective and $c_1\circ\pi'\circ(\id,\pi)=q\circ c$, it follows that $c_1\circ\pi'=q\circ c'$. Since $c_1$ corresponds to the semiabelian subvariety $G_1\subset G$, we see $G_1\subset H$.
\end{proof}

\begin{prop}
\label{generic relative to a f.g. group}
Let $G_1$ and $G_2$ be semiabelian varieties defined over an algebraically closed field $K$ of characteristic $0$, let $G=G_1\oplus G_2$ and $\pi_i\colon G\to G_i$ be the natural projection maps. If $\Gamma\subset G_2(K)$ is a finitely generated subgroup, then there exists $x_1\in G_1(K)$ with the following property: for any proper algebraic subgroup $H\subset G$ and for any $\gamma\in\Gamma$, if $(x_1,\gamma)\in H$ then $\pi_2(H)$ is a proper algebraic subgroup of $G_2$.
\end{prop}

In our proof for Proposition~\ref{generic relative to a f.g. group} we will use the following related result.

\begin{lemma}
\label{generic relative to a f.g. group torus}
Let $T$ be an algebraic torus, let $T_0\subset T$ be a subtorus, and let $\Gamma_0\subset T(K)$ be a finitely generated subgroup. Then there exists $y_0\in T_0(K)$ such that given any algebraic subgroup $H_0\subset T$, if there exists $\gamma_0\in\Gamma_0$ such that $y_0\cdot \gamma_0\in H_0(K)$ then $T_0\subset H_0$.
\end{lemma}

\begin{proof}
Since $K$ is algebraically closed, then $T$ splits and so, without loss of generality, we may assume $T=\bG_m^n$ and $T_0=\bG_m^{n_0}$ for some integers $n_0\le n$.  We let $\Gamma_{0,0}\subset \bG_m(K)$ be the finitely generated subgroup spanned by all the coordinates of a finite set of generators of $\Gamma_0$. Then we simply pick $y_0:=(y_{0,1},\dots, y_{0,n_0})\in \bG_m^{n_0}(K)$ with the property that for any nontorsion $\gamma_{0,0}\in\Gamma_{0,0}$ (i.e., $\gamma_{0,0}$ is not a root of unity) we have that $y_{0,1},\dots, y_{0,n_0},\gamma_{0,0}$ are multiplicatively independent. Since $\Gamma_{0,0}$ has finite rank, while $\bG_m(K)$ has infinite rank, we can always do this. 

Now, any algebraic subgroup $H_0\subset \bG_m^n$ is the zero locus of finitely many equations of the form 
\begin{equation}
\label{torus 0}
x_1^{m_1}\cdot \cdots \cdot x_n^{m_n}=1.
\end{equation}
for some integers $m_1,\dots, m_n$. 
Now, if there exists some $\gamma_{0}\in\Gamma_0$ such that $y_0\cdot \gamma_0\in H_0(K)$, then \eqref{torus 0} yields that 
\begin{equation}
\label{torus 1}
y_{0,1}^{m_1}\cdots y_{0,n_0}^{m_{n_0}}\in \Gamma_{0,0}.
\end{equation}
Our choice of $y_{0,1},\dots, y_{0,n_0}$ yields that $m_1=\cdots = m_{n_0}=0$; therefore $\bG_m^{n_0}\subset H_0$, as desired.
\end{proof}

\begin{proof}[Proof of Proposition~\ref{generic relative to a f.g. group}.]
We first observe that it is enough to prove the desired conclusion when each $G_i$ is replaced by a finite cover.
\begin{lemma}
\label{it suffices for finite cover lemma}
It suffices to prove Proposition~\ref{generic relative to a f.g. group} after replacing each $G_i$ (for $i=1,2$) by a finite cover.
\end{lemma}

\begin{proof}[Proof of Lemma~\ref{it suffices for finite cover lemma}.]
For each $i=1,2$, we let $\widetilde{G}_i$ be a semiabelian variety, let $\sigma_i:\widetilde{G}_i\lra G_i$ be an isogeny, and $\widetilde{G}:=\widetilde{G}_1\oplus \widetilde{G}_2$. We also let $\sigma:=(\sigma_1,\sigma_2):\widetilde{G}\lra G$ and $\widetilde{\pi}_i:\widetilde{G}\lra \widetilde{G}_i$ be the natural projections maps onto each coordinate. 

Since $\sigma_2$ is an isogeny, $\widetilde{\Gamma}:=\sigma_2^{-1}(\Gamma)$ is a finitely generated subgroup of $\widetilde{G}_2(K)$. We assume that the conclusion of Proposition~\ref{generic relative to a f.g. group} holds for $\widetilde{G}=\widetilde{G}_1\oplus \widetilde{G}_2$ and $\widetilde{\Gamma}$. Thus there exists $\widetilde{x}_1\in \widetilde{G}_1(K)$ such that for any proper algebraic subgroup $\widetilde{H}$ of $\widetilde{G}$, if there exists some $\widetilde{\gamma}\in \widetilde{\Gamma}$ with $(\widetilde{x}_1,\widetilde{\gamma})\in H(K)$, then $\widetilde{\pi}_2(\widetilde{H})$ is a proper algebraic subgroup of $\widetilde{G}_2$. We claim that $x_1:=\sigma_1(\widetilde{x}_1)\in G_1(K)$ satisfies the conclusion of Proposition~\ref{generic relative to a f.g. group}.

Indeed, assume there exists some proper algebraic subgroup  $H$ of $G$ containing $(x_1,\gamma)$ for some $\gamma\in \Gamma$. Then letting $\widetilde{H}:=\sigma^{-1}(H)$, we see that $(\widetilde{x}_1,\widetilde{\gamma})\in \widetilde{H}(K)$ and moreover, since $\sigma$ is an isogeny, $\widetilde{H}$ is also a proper algebraic subgroup of $\widetilde{G}$. Using the property satisfied by $\widetilde{x}_1$, it follows that $\widetilde{\pi}_2(\widetilde{H})$ is a proper algebraic subgroup of $\widetilde{G}_2$. Since $\pi_2\circ \sigma = \sigma_2\circ \widetilde{\pi}_2$ and $\sigma_2$ is an isogeny, we see $\pi_2(H)$ must be a proper algebraic subgroup of $G_2$, as desired.
\end{proof}

For $i=1,2$, we let 
$$0\lra T_i\lra G_i\stackrel{p_i}\lra A_i\lra 0$$ 
be a short exact sequence where $T_i$ is an algebraic tori and $A_i$ is an abelian variety. We also let $T:=T_1\times T_2$ and let $p:=(p_1,p_2):G\lra A$, where $A:=A_1\times A_2$. 

Using Lemmas~\ref{min dominant subgp exists after finite cover} and \ref{it suffices for finite cover lemma}, after replacing $G_1$ and $G_2$ by finite covers, if necessary, we can assume that $G$ admits 
 a minimal dominant semiabelian subvariety $H_0$.

We let $\overline{\Gamma}:=p_2(\Gamma)\subset A_2(K)$. Then applying \cite[Lemma~5.5]{GS-abelian}, there exists $\overline{x}_1\in A_1(K)$ with the following property: given any algebraic subgroup $\overline{H}\subset A=A_1\oplus A_2$ for which there exists some $\overline{\gamma}\in\overline{\Gamma}$ with $(\overline{x}_1,\overline{\gamma})\in \overline{H}$, we must have that $A_1\subset \overline{H}$.

We let $x_{1,0}\in G_1(K)$ such that $p_1(x_{1,0})=\overline{x}_1$.  We let $f:G\lra G/H_0$; since $p(H_0)=A$, then $G/H_0$ is an algebraic torus by Fact~\ref{torus not abelian 2}. We let $\Gamma '$ be the finitely generated subgroup of $G(K)$ spanned by $x_{1,0}$ and $\Gamma$ and let $\Gamma_0:=f(\Gamma ')$. We also let $U:=G/H_0$ and $U_0\subset U$ be the algebraic subtorus $f(T_1)$. According to Lemma~\ref{generic relative to a f.g. group torus}, there exists $y_0\in U_0(K)$ such that for any algebraic subgroup $V\subset U$, if there exists $\gamma_0\in\Gamma_0$ such that $y_0+\gamma_0\in V(K)$ then we must have that $U_0\subset V$. We let $t_0\in T_1(K)$ such that $f(t_0)=y_0$; we show next that $x_1:=t_0+x_{1,0}$ satisfies the conclusion of Proposition~\ref{generic relative to a f.g. group}.

So, let $H\subset G=G_1\oplus G_2$ be a proper algebraic subgroup containing  $(x_1,\gamma)$ for some $\gamma\in\Gamma$. We argue by contradiction and therefore assume $\pi_2(H)=G_2$. Since $p_1(x_1)=p_1(x_{1,0})=\overline{x}_1$, we obtain that $p(H)$ is an algebraic subgroup of $A$ containing $(\overline{x}_1,p_2(\gamma))$. Notice that $p_2(\gamma)\in\overline{\Gamma}$. If $p(H)$ were a proper subgroup of $A=A_1\oplus A_2$, then the hypothesis satisfied by $\overline{x}_1$ shows that $\overline{\pi}_{2}(p(H))$ is a proper algebraic subgroup of $A_2$, where $\overline{\pi}_{2}:A\lra A_2$ is the projection of $A=A_1\oplus A_2$ onto its second factor. However, $p_2(\pi_2(H))=\overline{\pi}_{2}(p(H))$ which contradicts our assumption that $\pi_2(H)=G_2$; it follows that $p(H)=A$. Using the minimality of $H_0$, we get that $H_0\subset H$.

Next we consider the projection map $f:G\lra G/H_0=U$. We have
$$f(x_1\oplus\gamma)=f(t_0)+f(x_{1,0})+f(\gamma)=y_0+f(x_{1,0})+f(\gamma)\in y_0+\Gamma_0;$$
on the other hand, $x_1\oplus \gamma\in H$ so $y_0+f(x_{1,0})+f(\gamma)=f(x_1\oplus \gamma)$ is contained in the subgroup $V:=f(H)$ of $U$. Our choice of $y_0$ yields that $U_0=f(T_1)\subset V$; taking inverse images under $f$, we have $T_1\subset H+H_0=H$. Since we also know that $p(H)=A=A_1\times A_2$, we see from Lemma \ref{conditions for containing G1} that $G_1\subset H$.

Finally, since $H$ is a proper algebraic subgroup of $G=G_1\times G_2$ containing $G_1$ (as shown above) and also projecting dominantly onto $G_2$ under the natural projection map $\pi_2$ (according to our assumption), we obtain a contradiction. Therefore $\pi_2(H)$ must be a proper algebraic subgroup of $G_2$. This concludes our proof of Proposition~\ref{generic relative to a f.g. group}. 
\end{proof}

\subsection{Constructing topological generators}
\label{Constructing topological generators}

The following is the main result of this subsection.

\begin{prop}
\label{existence of a point 2}
Let $K$ be an algebraically closed field of characteristic $0$.  
Let $\psi:B\lra C$ be a group homomorphism of  semiabelian varieties defined over $K$, and let $y\in C(K)$. If the algebraic subgroup generated by $\psi(B)$ and $y$ is $C$ itself, then there exists $x\in B(K)$ such that the Zariski closure of the cyclic subgroup generated by $\psi(x)+y$ is $C$.
\end{prop} 

We first prove a variant of Proposition~\ref{existence of a point 2}: when $B$ is an algebraic torus, but the algebraic group generated by $\psi(B)$ and $y$ is not necessarily equal to $C$. This result, proven in Proposition~\ref{existence of a point 1}, will then be used to derive Proposition~\ref{existence of a point 2}.

\begin{prop}
\label{existence of a point 1}
Let $K$ be an algebraically closed field of characteristic $0$, let $T$ be an algebraic torus and $C$ be a semiabelian variety.   
Let $\psi:T\lra C$ be a homomorphism of algebraic groups defined over $K$, and let $y\in C(K)$. Then there exists $x\in T(K)$ such that the Zariski closure of the cyclic subgroup generated by $\psi(x)+y$ is the 
algebraic group generated by $\psi(T)$ and $y$.
\end{prop}

\begin{proof}
Our argument follows the proof of \cite[Lemma~5.1]{GS-abelian}.

Let $K_0$ be a finitely generated subfield of $K$ such that $T$, $C$, and  $\psi$ are defined over $K_0$, and moreover, $y\in C(K_0)$. So, without loss of generality, we may assume $K$ is the algebraic closure of $K_0$.  

We let $T=T_1\oplus\cdots \oplus T_m$ be written as a direct sum of $1$-dimensional algebraic tori; at the expense of replacing $K_0$ by a finite extension, we may assume each $T_i$ is defined over $K_0$. Then $$\psi(T)=\sum_{i=1}^m \psi(T_i)$$ 
and moreover, each $\psi(T_i)$ is either trivial or a $1$-dimensional algebraic torus. Our strategy is to find an algebraic point $z_i\in \psi(T_i)$  
such that if $z:=\sum_{i=1}^m z_i$, then the Zariski closure of the cyclic group generated by $z+y$ is the algebraic group 
generated by $\psi(T)$ and $y$.  If for some $i$ we have that $\psi(T_i)=\{0\}$ is trivial, then we simply pick $z_{i}=0$. Now consider those $i\in\{1,\dots, m\}$ such that $\psi(T_i)$ is nontrivial. For each such $i$, we will show there exist $z_{i}\in \psi(T_i)$ such that for any positive integer $n$ we have
\begin{equation}
\label{necessary technical condition}
nz_{i}\notin \left(\psi(T_i)\right)\left(K_0\left(C_{\tor},  z_{1},\dots, 
z_{i-1}\right)\right).
\end{equation}

\begin{claim}
\label{one technical claim}
If the above condition \eqref{necessary technical condition} holds for each $i=1,\dots, m$ such that $\psi(T_i)\ne \{0\}$, then the Zariski closure of the cyclic group generated by 
$z+y$ is the algebraic subgroup generated by $\psi(T)$ and $y$.
\end{claim}

\begin{proof}[Proof of Claim~\ref{one technical claim}.]
First, we note that if \eqref{necessary technical condition} holds, then $z_i\ne 0$; therefore, $z_i=0$ automatically implies that $\psi(T_i)=\{0\}$.

Now, assume 
there exists some algebraic subgroup $D\subset C$ (not necessarily connected) such that 
$z+y\in D(K)$. Let $i\le m$ be the largest integer such that $z_{i}\ne 0$; then we have
$$z_{i}\in \left((-y-z_{1}-\cdots -z_{i-1})+D\right)\cap \psi(T_i).$$
Assume first that $\psi(T_i)\cap D$ is a proper algebraic subgroup of $\psi(T_i)$. Since 
$\psi(T_i)$ is a $1$-dimensional torus, we see $D\cap \psi(T_i)$ is a $0$-dimensional 
algebraic subgroup of $C$; hence there exists a nonzero integer $n$ such that 
$n\cdot (D\cap \psi(T_i))=\{0\}$. Then $nz_{i}$ is the only (geometric) point of the 
subvariety $n\cdot \left(\left((-y-z_{1}-\cdots -z_{i-1})+D\right)\cap \psi(T_i)\right)$ 
which is thus rational over $K_0(C_{\tor},z_{1},\dots, z_{i-1})$. But by our 
construction, $$nz_{i}\notin \psi(T_i)(K_0(C_{\tor}, z_{1},\dots, z_{i-1}))$$ 
which is a contradiction. Therefore $\psi(T_i)\subset D$ if $i$ is the largest index in $\{1,2,\dots,m\}$ such that $z_{i}\ne 0$, or equivalently, if $i$ is the largest index for which $\psi(T_i)\ne 0$. 

Now note that $z+y=\sum_{j\leq i}z_j+y$ and $z_i\in \psi(T_i)\subset D$ so $z'+y=z+y-z_i\in D$, where $z':=z_{1}+\cdots +z_{i-1}$. Repeating the exact same argument as above for the next positive integer $i_1<i$ for which $\psi(T_{i_1})\ne \{0\}$, and then arguing inductively we obtain that each $\psi(T_j)$ is contained in $D$, and therefore $\psi(T)\subset D$. But then 
$z\in \psi(T)\subset D$ and so, $y\in D$ as well, which yields that  the Zariski 
closure of the cyclic group generated by $z+y$ is the algebraic subgroup of $C$ generated by $\psi(T)$ 
and $y$, as desired.  
\end{proof}

We just have to show that we can choose $z_{i}$ satisfying \eqref{necessary technical 
condition}. 
So, the problem reduces to the following: $L$ is a finitely generated field of characteristic 
$0$, $\varphi$ is an algebraic group homomorphism between an algebraic torus $U$ and some semiabelian variety $C$ all defined over $L$, $\varphi$ has finite kernel, and we want to find $x\in U(\overline{L})$ such that for each positive integer $n$, we have
\begin{equation}
\label{necessary technical condition 2}
n\varphi(x)\notin \varphi(U)\left(L\left(C_{\tor}\right)\right).
\end{equation}
Indeed, with the above notation, $U:=T_i$ (for each $i=1,\dots, m$), $L$ is the extension of $K_0$ generated  by $z_{j}$ 
(for $j=1,\dots, i-1$), and $\varphi$ is the homomorphism 
$\psi$ restricted to $U=T_i$ for which $\psi(T_i)$ is nontrivial. 

Let $d$ be the degree of the isogeny $\varphi ':U\lra \varphi(U)\subset C$. In particular, this means that for each $z\in C(\overline{L})$ and each $x\in U(\overline{L})$ for which $\varphi(x)=z$ we have
\begin{equation}
\label{key index inequality}
\left[L(x):L\right]\le d\cdot \left[L(z):L\right].  
\end{equation}
For any subfield $M\subset \overline{L}$, we let 
$M^{(d)}$ be the compositum of all extensions of $M$ of degree at most equal to $d$. 

\begin{claim}
\label{simple groups claim}
Let $L$ be a finitely generated field of characteristic $0$, let $C$ be a  semiabelian variety defined over $L$, let $L_{\tor}:=L(C_{\tor})$, and let $d$ be a positive integer. Then there exists a normal extension of $L_{\tor}^{(d)}$ whose Galois group is not abelian.
\end{claim}

\begin{proof}[Proof of Claim~\ref{simple groups claim}.]
The proof is identical to the one from \cite[Claim~5.3]{GS-abelian}. Note that $L(C_{\tor})$ is Hilbertian since we can still apply \cite[Theorem,~p.~238]{Thornhill} due to Fact~\ref{torsion extensions}.
\end{proof}

Claim~\ref{simple groups claim} yields that there exists a point $x\in U(\overline{L})$ which is not defined over an abelian extension of $L\left(C_{\tor}\right)^{(d)}$; i.e., $nx\notin U\left(L\left(C_{\tor}\right)^{(d)}\right)$ for all positive integers $n$. Hence, $n\varphi(x)\notin \varphi(U)\left(L\left(C_{\tor}\right)\right)$ (see \eqref{key index inequality}), which  concludes the proof of Proposition~\ref{existence of a point 1}.
\end{proof}

\begin{proof}[Proof of Proposition~\ref{existence of a point 2}.]
Let 
$$0\lra T_1\lra B\stackrel{p_1}\lra A_1\lra 0$$
$$0\lra T_2\lra C\stackrel{p_2}\lra A_2\lra 0$$ 
be two short exact sequences of algebraic groups with $T_i$ tori and $A_i$ abelian varieties. We let $\overline{y}:=p_2(y)$. By Fact~\ref{morphism fact}, the endomorphism $\psi:B\lra C$ induces an endomorphism of abelian varieties $\overline{\psi}:A_1\lra A_2$. Using \cite[Lemma~5.1]{GS-abelian}, we conclude that there exists $x_0\in A_1(K)$ such that the Zariski closure of the cyclic group generated by $\overline{\psi}(x_0)+\overline{y}$ equals the algebraic subgroup generated by $\overline{\psi}(A_1)$ and $\overline{y}$. Since the algebraic subgroup generated by $\psi(B)$ and $y$ equals $C$, we conclude that the algebraic subgroup generated by $\overline{\psi}(A_1)$ and $\overline{y}$ equals $A_2$. So, the cyclic subgroup generated by $\overline{\psi}(x_0)+\overline{y}$ is Zariski dense in $A_2$.

Choose a point $x_1\in B(K)$ such that $p_1(x_1)=x_0$ and let $y_1:=\psi(x_1)+y\in C(K)$. Using Proposition~\ref{existence of a point 1}, we can find $t\in T_1(K)$ such that the Zariski closure $H$ of the cyclic group generated by $\psi(t)+y_1$ is equal to the algebraic group generated by $\psi(T_1)$ and $y_1$. We claim that the point $x:=x_1+t$ satisfies the conclusion of Proposition~\ref{existence of a point 2}. Since $\psi(x)+y=\psi(t)+\psi(x_1)+y=\psi(t)+y_1$, it therefore suffices to prove the following:

\begin{lemma}
\label{T and y_1 is all}
With the above notation, $H=C$.
\end{lemma}

\begin{proof}[Proof of Lemma~\ref{T and y_1 is all}.]
We let $U$ be the algebraic subgroup which is the Zariski closure of the cyclic group generated by $y_1$. By our choice of $x_0$, $x_1$ and $t$, we know that
\begin{itemize}
\item[(i)] $\psi(T_1)\subset H$;
\item[(ii)] $U\subset H$; and
\item[(iii)] $p_2(U)=A_2$.
\end{itemize}

Statements (i) and (ii) follow directly from the definitions. Statement (iii) holds because $p_2(y_1)=p_2(\psi(x_1))+p_2(y)=\overline\psi (p_1(x_1))+p_2(y)=\overline\psi(x_0)+\overline{y}$ and by the fact that the Zariski closure of the cyclic group generated by $\overline\psi(x_0)+\overline{y}$ equals $A_2$. Our hypothesis that the algebraic subgroup generated by $\psi(B)$ and $y$ is $C$ itself yields that $\psi(B)+U=C$. Our goal is to show that $\psi(T_1)+U=C$.

Using property~(iii) above and Fact~\ref{torus not abelian 2}, we see $C/U$ is an algebraic torus. Since $\psi(B)/(\psi(B)\cap U)\simeq(\psi(B)+U)/U=C/U$, we see
\begin{equation}
\label{a torus 1}
\psi(B)/(\psi(B)\cap U)\text{ is an algebraic torus.}
\end{equation}
Since $\psi(T_1)$ is the toric part of $\psi(B)$, we obtain that 
\begin{equation}
\label{a torus 2}
\psi(T_1)/(\psi(T_1)\cap U)\text{ is the toric part of }\psi(B)/(\psi(B)\cap U).
\end{equation}

Equations \eqref{a torus 1} and \eqref{a torus 2} yield that 
\begin{equation}
\label{a torus 3}
\psi(B)/(\psi(B)\cap U) \isomto \psi(T_1)/(\psi(T_1)\cap U)\text{ and therefore,}
\end{equation}
\begin{equation}
\label{a torus 4}
(\psi(B)+U)/U\isomto (\psi(T_1)+U)/U.
\end{equation}
Equation \eqref{a torus 4} yields that $\dim(\psi(B)+U)=\dim(\psi(T_1)+U)$ and because $C=\psi(B)+U$ is connected, we conclude that $H=\psi(T_1)+U=C$, as desired. 
\end{proof}
This concludes our proof of Proposition~\ref{existence of a point 2}.
\end{proof}

\subsection{Conditions to guarantee the existence of a Zariski dense orbit}
\label{Conditions to guarantee a point has Zariski dense orbit}

\begin{lemma}
\label{combinatorial lemma}
Let $K$ be an algebraically closed field of characteristic $0$, let $G$ be a semiabelian variety defined over $K$, let   $y_1,\dots, y_r\in G(K)$, and let $P_1,\dots, P_r\in \Q[z]$ such that $P_i(n)\in\Z$ for 
each $n\ge 1$ and for each $i=1,\dots, r$, while $\deg(P_r)>\cdots >\deg(P_1)>0$. For an  
infinite subset $S\subset \N$, let $V:=V(S; P_1,\dots , P_r; y_1, \dots, y_r)$ be the Zariski closure of the set 
$$\left\{P_1(n)y_1+ \cdots + P_r(n)y_r\colon n\in S\right\}.$$ 
Then 
there exist nonzero integers $\ell_1,\dots, \ell_r$ such that  $V$ 
contains  a coset of the subgroup $\Gamma$ generated by $\ell_1y_1,\dots, \ell_ry_r$.  
\end{lemma}

\begin{proof}
The proof is almost identical with the proof of \cite[Lemma~5.6]{GS-abelian}; however, since that proof employed (though, in a non-essential way) Poincar\'e's Reducibility Theorem for abelian varieties, we include a proof for our present lemma in the context of semiabelian varieties which, of ocurse, does not use the Poincar\'e's Reducibility Theorem.

Let $\Gamma_0$ be the subgroup of $G$ generated  by $y_1,\dots, y_r$. Since $V(K)\cap\Gamma_0$ is Zariski dense in $V$, then by Fact~\ref{Faltings theorem} we see that $V$ is a finite union of cosets of algebraic subgroups of $G$.  So, at the expense of replacing $S$ by an infinite subset, we may assume $V=z+C$, for some $z\in G(K)$ and some irreducible algebraic subgroup $C$ of $G$. Hence $\{-z+P_1(n)y_1+\cdots + P_r(n)y_r\}_{n\in S}\subset C(K)$. We will show there exist nonzero integers $\ell_i$ such that $\ell_iy_i\in C(K)$ for each $i=1,\dots, r$.

We proceed by induction on $r$. We first handle the base case when $r=1$. Then $\{P_1(n)\}_{n\in S}$ takes infinitely many distinct integer values as $\deg(P_1)\ge 1$, and in particular there exist $n_0,n\in S$ with $\ell:=P_1(n)-P_1(n_0)$ non-zero. Since $C(K)$ is a subgroup of $G(K)$, we see $\ell y_1=(-z+P_1(n)y_1)-(-z+P_1(n_0)y_1)\in C(K)$.

Next let $s\ge 2$. Assume the statement holds for all $r<s$, we prove it for $r=s$. Let $n_0\in S$. Letting $P'_i:=P_i-P_i(n_0)$, we see $\{P'_1(n)y_1+\cdots + P'_s(n)y_s\}_{n\in S}\subset C(K)$. Since $\deg(P'_1)\ge 1$ there exists $n_1\in S$ such that $P'_1(n_1)\ne 0$. For each $i=2,\dots, s$ we let
$$Q_i(z):=P'_1(n_1) P'_i(n)-P'_1(n) P'_i(n_1).$$
Since $C(K)$ is a subgroup of $G(K)$ and $\sum_{i=2}^s P'_i(n)y_i\in C(K)$ it follows that $\sum_{i=2}^s P'_i(n)P'_1(n_1)y_i\in C(K)$. Similarly, $\sum_{i=2}^s P'_1(n)P'_i(n_1)y_i\in C(K)$. Subtracting, we have
$$\left\{\sum_{i=2}^s Q_i(n)y_i\right\}_{n\in S}\subset C(K).$$ 
Since $\deg(Q_i)=\deg(P_i)$ for each $i=2,\dots, s$, we can use the induction hypothesis and conclude that there exist nonzero integers $\ell_2,\dots, \ell_s$ such that $\ell_iy_i\in C(K)$ for each $i\ge 2$. Let $\ell_1:=P'_1(n_1)\cdot\prod_{i=2}^s\ell_i$ which is non-zero since $P'_1(n_1)$ is. Since $P'_1(n_1)y_1+\cdots +P'_s(n_1)y_s\in C(K)$, we see $\ell_1y_1=\left(P'_1(n_1)\cdot\prod_{i=2}^s\ell_i \right)y_1\in C(K)$. This concludes our proof.
\end{proof}

Lemma~\ref{combinatorial lemma} has the following important consequence for us.
\begin{lemma}
\label{infinitely many iterates nilpotent operator}
Let $K$ be an algebraically closed field of characteristic $0$, let $G$ be a semiabelian variety defined over $K$, let $\tau \in \End(G)$ with the property that there exists a positive integer $r$ such that $(\tau-\id)^r=0$, let $y\in G(K)$, let $\varphi:G\lra G$ be a self-map such that $\varphi(x)=\tau(x)+y$ for each $x\in G$. 

Let $x\in G(K)$ and let $c+C$ be a coset of an algebraic subgroup $C\subset G$ with the property that there exists an infinite set $S$ of positive integers such that  $\{\varphi^n(x)\colon n\in S\}\subset c+C$. Then there exists a positive integer $\ell$ such that $\ell\cdot\left(\beta(x)+ y\right)\in C(K)$, where $\beta:=\tau-\id$. 

Moreover, if  the cyclic group generated by $\beta(x)+y$ is Zariski dense in $G$, then $C=G$ and therefore, the set $\{\varphi^n(x)\colon n\in S\}$ is Zariski dense in $G$.
\end{lemma}

\begin{proof}
The proof is identical with the derivation of \cite[Lemma~5.7]{GS-abelian} from \cite[Lemma~5.6]{GS-abelian}; this time, one employs Lemma~\ref{combinatorial lemma} in order to derive the desired conclusion.

For the ``moreover'' part in Lemma~\ref{infinitely many iterates nilpotent operator}, one argues as in the proof of \cite[Corollary~5.8]{GS-abelian}; note that if a cyclic subgroup of $G(K)$ is Zariski dense, then any infinite subgroup of it is also Zariski dense (see also \cite[Lemma~3.9]{GS-abelian}).
\end{proof}


\section{Proof of our main result}
\label{section proof}

\begin{proof}[Proof of Theorem~\ref{main result}.]
By Fact~\ref{rigidity}, there exists a dominant group endomorphism $\tau :G\lra G$, and there exists $y\in G(K)$ such that $\varphi(x)=\tau(x)+y$ for all $x\in G$. By \cite[Lemma~2.1]{preprint 2}, it suffices to prove Theorem~\ref{main result} for an iterate $\varphi^n$ with $n>0$. Replacing $\varphi$ by $\varphi^n$ replaces $y$ by $\sum_{i=0}^{n-1} \tau^i(y)$ and $\tau$ by $\tau^n$. As a result, we may assume 
\begin{equation}
\label{kernel tau}
\dim\ker(\tau^m-\id)=\dim(\ker(\tau-\id))\text{ for all $m\in\N$.} 
\end{equation}
Letting $f\in\Z[t]$ be the minimal polynomial of $\tau\in\End(G)$, we may therefore assume that $1$ is the only root of unity which is a root of $f$.

Let $r$ be  the order of vanishing at $1$ of $f$, and let $f_1\in \Z[t]$ such that $f(t)=f_1(t)\cdot (t-1)^r$. Then $f_1$ is also a monic polynomial. 
Let $G_1:=(\tau-\id)^r(G)$ and let $G_2:=f_1(\tau)(G)$, where $f_1(\tau)\in\End(G)$ and $\id$ is the identity map on $G$. 
By definition, both $G_1$ and $G_2$ are connected algebraic subgroups of $G$, hence they are both semiabelian subvarieties of $G$. By definition, the restriction $\tau|_{G_1}\in\End(G_1)$ has  minimal polynomial equal to $f_1$ whose roots are not roots of unity. On the other hand, $(\tau-\id)^r|_{G_2}=0$. Furthermore, as shown in \cite[Lemma~6.1]{GS-abelian},
\begin{equation} 
\label{decomposition G}
G=G_1+G_2\text{ and }G_1\cap G_2\text{ is finite.}
\end{equation}
Even though \cite[Lemma~6.1]{GS-abelian} was written in the context of abelian varieties, it uses no specific properties of abelian varieties; instead it is valid for any commutative algebraic group. So, $G$ is isogenuous with the direct product $G_1\times G_2$.

We let $y_1\in G_1$ and $y_2\in G_2$ such that $y=y_1+y_2$. We  denote by $\tau_i$ the induced action of $\tau$ on each $G_i$. Since the minimal polynomial $f_1$ of $\tau_1\in\End(G_1)$ does not have the root $1$, it follows that $(\id-\tau_1):G_1\lra G_1$ is an isogeny. As a result, there exists $y_0\in G_1(K)$ such that $(\id-\tau_1)(y_0)=y_1$. Using Lemma~\ref{replacing by a conjugate}, it suffices to prove Theorem~\ref{main result} for $T_{-y_0}\circ \varphi\circ T_{y_0}$, where $T_z$ represents the translation-by-$z$ automorphism of $G$ (for any given point $z\in G$). We may therefore assume that $y_1=0$.

Let $\varphi_i:G_i\lra G_i$ be given by $\varphi_1(x)=\tau_1(x)$ and $\varphi_2(x)=\tau_2(x)+y_2$; then for each $x_1\in G_1(K)$ and $x_2\in G_2(K)$ we have that 
\begin{equation}
\label{decomposition varphi}
\varphi(x_1+x_2)=\varphi_1(x_1)+\varphi_2(x_2). 
\end{equation}
We let $\beta:=(\tau_2-\id)|_{G_2}\in\End(G_2)$; then $\beta^r=0$. Let $B$ be the Zariski closure of the subgroup of $G_2$ generated by $\beta(G_2)$ and $y_2$; then $B$ is an algebraic subgroup of $G_2$.

\begin{lemma}
\label{fibration exists}
Assume $B\ne G_2$. Then $C:=G_1+B$ is a proper algebraic subgroup of $G$ and moreover, letting $f:G\lra G/C$ be the natural quotient homomorphism, we have that $f\circ \varphi=f$.
\end{lemma} 

\begin{proof}[Proof of Lemma~\ref{fibration exists}.]
Since $G_2$ is connected and $B$ is assumed to be a proper algebraic subgroup, we have $\dim(B) <\dim(G_2)$. As a result, \eqref{decomposition G} tells us that $C=G_1+B$ is also a proper algebraic subgroup of $G$. Then the quotient map  $f:G\lra G/C$  
is a dominant morphism to a nontrivial semiabelian variety and moreover, we claim that $f\circ \varphi = f$. Indeed, for each $x\in G$, we let $x_i\in G_i$ (for $i=1,2$) such that $x=x_1+x_2$ (see \eqref{decomposition G}) and then we get  
\begin{align*}
f(\varphi(x)) & = f(\varphi_1(x_1)+\varphi_2(x_2))\text{ by \eqref{decomposition varphi}}\\
& = f(\varphi_2(x_2))\text{ because $\varphi_1(x_1)\in G_1\subset C$}\\
& = f(x_2+\beta(x_2)+y_2)\text{ by definition of $\varphi_2$ and $\beta$}\\
& = f(x_2)\text{ because $\beta(x_2),y_2\in B\subset C$}\\ 
& = f(x_1+x_2)\text{ because $x_1\in G_1$}\\
& = f(x),
\end{align*}
as desired.
\end{proof}

By Lemma~\ref{fibration exists}, if $B\neq G_2$ then $\varphi$ preserves a non-constant fibration and so Theorem~\ref{main result} holds. As a result, we may assume that $B=G_2$. We will prove in this case that there exists $x\in G(K)$ with a Zariski dense orbit under the action of $\varphi$. In order to do this, we first show that we may also assume $G$ is the direct product $G_1\oplus G_2$. Indeed, we construct 
$$\widetilde{\varphi}:=(\varphi_1,\varphi_2):G_1\oplus G_2\lra G_1\oplus G_2,$$ 
where (as before) $\varphi_1(x_1)=\tau_1(x_1)$ for each $x_1\in G_1$ and $\varphi_2(x_2)=\tau_2(x_2)+y_2$ for each $x_2\in G_2$. We also let $\sigma:G_1\oplus G_2\lra G$ given by $\sigma(x_1\oplus x_2)=\iota_1(x_1)+\iota_2(x_2)$, where $\iota_j:G_j\lra G$ are the inclusion maps.

\begin{lemma}
\label{it suffices for direct product}
If there exists $(x_1,x_2)\in (G_1\oplus G_2)(K)$ with a Zariski dense orbit under the action of $\widetilde{\varphi}$, 
then $x:=\sigma(x_1,x_2)\in G(K)$ has a Zariski dense orbit under $\varphi$.
\end{lemma}

\begin{proof}[Proof of Lemma~\ref{it suffices for direct product}.]
Indeed, identifying each $G_j$ with its image $\iota_j(G_j)$ inside $G$,  \eqref{decomposition varphi} yields that 
\begin{equation}
\label{commuting sigma}
\sigma\circ \widetilde{\varphi}=\varphi\circ \sigma. 
\end{equation}
Then equation \eqref{commuting sigma} yields $\sigma\circ \widetilde{\varphi}^n=\varphi^n\circ \sigma$ for each $n\in\N$, which means that if there exists a Zariski dense orbit $\OO_{\widetilde{\varphi}}(x_1\oplus x_2)\subset (G_1\oplus G_2)(K)$, then $\OO_\varphi(x_1+x_2)\subset G(K)$ is also a Zariski dense orbit; note that $\sigma$ is a dominant homomorphism.  
\end{proof}

So, from now on, we may assume $G=G_1\oplus G_2$ and that $\varphi:G\lra G$ is given by the action $(x_1,x_2)\mapsto (\varphi_1(x_1),\varphi_2(x_2))$. 

In order to prove the existence of a $K$-point in $G$ with a Zariski dense orbit, we first prove there exists $x_2\in G_2(K)$ such that $\OO_{\varphi_2}(x_2)$ is Zariski dense in $G_2$. Since we assumed that the group generated by $\beta(G_2)$ and $y_2$ is Zariski dense in $G_2$, Proposition~\ref{existence of a point 2} yields the existence of $x_2\in G_2(K)$ such that the cyclic group generated by $\beta(x_2)+y_2$ is Zariski dense in $G_2$. Then Lemma~\ref{infinitely many iterates nilpotent operator} yields that any infinite subset of $\OO_{\varphi_2}(x_2)$ is Zariski dense in $G_2$. If $G_1$ is trivial, then $G_2=G$ and $\varphi=\varphi_2$ and so,  Theorem~\ref{main result} is proven. Hence, from now on, assume that $\dim(G_1)>0$.

Let $\pi_i$ (for $i=1,2$) be the projection of $G$ onto each of its two factors $G_i$. 
Let $\Gamma$ be the $\End(G_2)$-submodule of $G_2(K)$ generated by $x_2$ and $y_2$. By Fact~\ref{endomorphism f.g.}, $\Gamma$ is a finitely generated subgroup of $G_2(K)$.  Using Proposition~\ref{generic relative to a f.g. group}, we may find $x_1\in G_1(K)$ with the property that if there exists a proper algebraic subgroup $H\subset G=G_1\oplus G_2$ such that $x_1\in \Gamma+H$ (or equivalently, there exists $\gamma\in\Gamma$ such that $(x_1,\gamma)\in H\subset G_1\oplus G_2$), then $\pi_2(H)$ is a proper algebraic subgroup of $G_2$. Let $x:=x_1\oplus x_2$ (or equivalently, $x=(x_1,x_2)$); we will prove that $\OO_\varphi(x)$ is Zariski dense in $G$.

Let $V$ be the Zariski closure of $\OO_\varphi(x)$. Then Fact~\ref{fact closure orbit} yields that $V$ is a finite union of cosets of algebraic subgroups of $G$. So, if $V\ne G$, then there exists a coset $c+H$ of a proper algebraic subgroup $H\subset G$ which contains  $\{\varphi^n(x)\}_{n\in S}$ for some infinite subset $S\subset \N$. In particular, for any integers $n>m$ from $S$, we have that 
\begin{equation}
\label{inclusion 0}
H\text{ contains  }\varphi^n(x)-\varphi^m(x). 
\end{equation}
Using the fact that $G=G_1\oplus G_2$, we construct $\mu:G\lra G$ as
$$\mu(z_1,z_2):=\left(\tau_1^n(z_1)-\tau_1^m(z_1), z_2\right).$$
Recall that the minimal polynomial $f_1$ of $\tau_1=\tau|_{G_1}$ does not have eigenvalues which are roots of unity, and so $\left(\tau_1^{n-m}-\id\right)$ is an isogeny on $G_1$. Because $\tau_1$ is also an isogeny on $G_1$, we see $\mu$ is an isogeny on $G$. Since
\begin{equation}
\label{inclusion 1}
\varphi^n(x)-\varphi^m(x)=\left(\tau_1^n(x_1)-\tau_1^m(x_1)\right)\oplus \left( \varphi_2^n(x_2) - \varphi_2^m(x_2)\right)
\end{equation}
$$\text{and }\varphi_2^n(x_2)-\varphi_2^m(x_2)\in\Gamma,$$
we obtain that there exists $\gamma\in \Gamma$ such that $\mu(x_1,\gamma)\in H$. In particular, this yields that $(x_1,\gamma)\subset \mu^{-1}(H)$; furthermore, $\mu^{-1}(H)$ is a proper algebraic subgroup of $G$ since $\mu$ is an isogeny. By our choice of $x_1$, we conclude that $\pi_2(\mu^{-1}(H))$ is a proper algebraic subgroup of $G_2$. However, since $\mu|_{G_2}$ is the identity map, we get that $\pi_2(H)$ is a proper algebraic subgroup of $G_2$. On the other hand, using \eqref{inclusion 0} and \eqref{inclusion 1}, we get that for any integers $n>m$ from $S$,  
\begin{equation}
\label{inclusion 2}
\pi_2(H)\text{ contains }\varphi_2^n(x_2)-\varphi_2^m(x_2).
\end{equation}
As a result, if we fix $m_0\in S$ we see that there are infinitely many $n$ for which $\varphi_2^n(x_2) - \varphi_2^{m_0}(x_2)\in H$. That is, the coset $\varphi_2^{m_0}(x_2) + \pi_2(H)$ contains infinitely many points of the form $\varphi_2^n(x_2)$. Notice that $\varphi_2(x)=\tau_2(x)+y_2$ and $\beta = \tau_2 - \id$ is nilpotent. Furthermore, the cyclic subgroup generated by $\beta(x_2) + y_2$ is Zariski dense in $G_2$. As a result, Lemma~\ref{infinitely many iterates nilpotent operator} tells us $G_2=\pi_2(H)$, which is a contradiction. Hence $\OO_\varphi(x)$ is Zariski dense in $G$, which concludes our proof.
\end{proof}

\begin{remark}
\label{more precise statement}
As shown in the proof of Theorem~\ref{main result} (see Lemma~\ref{fibration exists} specifically), we obtain that there exists a positive integer $n$ such that if $\varphi$ preserves a nonconstant fibration, then actually there exists a proper algebraic subgroup $C$ such that 
\begin{equation}
\label{H n}
f\circ \varphi^n=f\text{, where }f:G\lra G/C 
\end{equation}
is the usual quotient homomorphism. Also, one cannot expect that $n$ can be taken to be equal to $1$ in \eqref{H n}, as shown by the following example. If $\varphi:\bG_m\lra\bG_m$ is given by $\varphi(x)=1/x$, then $\varphi^2$ is the identity on $\bG_m$, and so, with the above notation, $n=2$ and $C=\{1\}$ is the trivial subgroup of $\bG_m$.  On the other hand, $\varphi$ does not preserve a nonconstant power map on $\bG_m$; instead $\varphi$ preserves the nonconstant rational function $f(x):=x+1/x$.

So, the most one can get for the self-map $\varphi$ itself is that there exists a finite collection of proper algebraic subgroups $C_1,\dots, C_\ell$ of $G$ such that if $\varphi$ preserves a nonconstant fibration, then each orbit of a point in $G$ is contained in a finite union of cosets $c_i+C_i$ (for some $c_i\in G$). The subgroups $C_i$ are precisely the subgroups appearing in the orbit under $\varphi$ of the subgroup $C$ from \eqref{H n}; note that equation \eqref{H n} yields that $C$ is fixed by $\varphi^n$ and so, there exist finitely many subgroups $C_i$ in the orbit of $C$ under the action of $\varphi$.
\end{remark}

\begin{remark}
\label{commutative monoid}
One could ask whether our arguments could be adapted to yield a generalization of Theorem~\ref{main result} in which the action of the cyclic monoid generated by $\varphi$ is replaced by the action of a finitely generated commutative monoid $S$ of regular self-maps on the semiabelian variety $G$. The corresponding statement for abelian varieties was proven in \cite[Theorem~1.3]{GS-abelian}, essentially using the same strategy as in the case of a cyclic monoid (i.e., \cite[Theorem~1.2]{GS-abelian}), combined with some results regarding commutative monoids and linear algebra. However, in the proof from \cite[Theorem~1.3]{GS-abelian} (see the bottom of \cite[page~462]{GS-abelian}), one uses Poincar\'e's Reducibility Theorem in a crucial way by finding a complement of a given algebraic subgroup of an abelian variety. In our proof of Theorem~\ref{main result} we can construct such a complement (see \eqref{decomposition G}) even in the absence of Poincar\'e's Reducibility Theorem, but that strategy fails when one deals with an arbitrary finitely generated commutative monoid $S$; choosing a decomposition of $G$ as a sum of two semiabelian subvarieties as in \eqref{decomposition G} which works simultaneously for all maps from $S$ is not possible unless either $S$ is cyclic (as in Theorem~\ref{main result}), or $G$ is a split semiabelian variety (and therefore Poincar\'e's Reducibility Theorem applies). So, for a non-split semiabelian variety $G$, in the absence of Poincar\'e's Reducibility Theorem, one would need a completely new strategy for proving the generalization of Theorem~\ref{main result} regarding a finitely generated commutative monoid of regular self-maps acting on $G$.
\end{remark}



\end{document}